\documentclass[a4paper,11pt,oneside]{article}
\usepackage[margin=1.2in]{geometry}

\usepackage{amsfonts}
\usepackage{dsfont}
\usepackage{amsmath,amssymb,amsthm}
\usepackage{algorithm,algpseudocode}
\usepackage[round]{natbib}

\DeclareMathOperator*{\argmax}{argmax}

\newcommand{\X}{\mathcal{X}}

\newcommand{\e}{e}
\def\eps{\varepsilon}

\usepackage{bbm}
\newcommand{\chr}{\boldsymbol{\mathbbm{1}}} %
\newcommand{\pred}[1]{\chr\left[ #1 \right]}

\newcommand{\poly}{\operatorname{poly}}

\newtheorem{theorem}{Theorem}
\newtheorem{lemma}[theorem]{Lemma}

\newcommand{\beq}{\begin{eqnarray*}}
\newcommand{\eeq}{\end{eqnarray*}}
\newcommand{\beqn}{\begin{eqnarray}}
\newcommand{\eeqn}{\end{eqnarray}}
\newcommand{\ben}{\begin{enumerate}}
\newcommand{\een}{\end{enumerate}}
\newcommand{\bit}{\begin{itemize}}
\newcommand{\eit}{\end{itemize}}
\newcommand{\hide}[1]{}

\newcommand{\nrm}[1]{\left\Vert #1 \right\Vert}

\newcommand{\R}{\mathbb{R}}

\newcommand{\N}{\mathbb{N}}

\newcommand{\paren}[1]{\left( #1 \right)}

\newcommand{\set}[1]{\left\{ #1 \right\}}

\newcommand{\abs}[1]{\left| #1 \right|}

\renewcommand{\P}{\mathop{\mathbb{P}}}
\newcommand{\E}{\mathop{\mathbb{E}}}

\newcommand{\Nclose}{N_{\mathsf{close}}}
\newcommand{\Nfar}{N_{\mathsf{far}}}
\newcommand{\Aclose}{A_{\mathsf{close}}}
\newcommand{\Afar}{A_{\mathsf{far}}}

\newcommand{\inv}{^{-1}} %
\newcommand{\Hquad}{\hspace{0.5em}}
\usepackage{algorithm} 
\usepackage{threeparttable}
\usepackage{caption}
\usepackage{blindtext}

\begin{document}

\title{Fast and Bayes-consistent nearest neighbors}

\author{Klim Efremenko \\\texttt{klim@bgu.ac.il} \and Aryeh Kontorovich \\\texttt{karyeh@cs.bgu.ac.il}
\and Moshe Noivirt \\\texttt{noivirt.moshe@gmail.com}
}

\maketitle

\begin{abstract}
Research on nearest-neighbor methods tends to focus
somewhat dichotomously
either
on
the statistical or the computational aspects
---
either on, say, 
Bayes consistency and rates
of convergence
or on techniques for speeding up the proximity search.
This paper aims at bridging these realms:
to reap the advantages of fast evaluation time
while maintaining Bayes consistency,
and further
without sacrificing too much in the risk decay rate.
We combine the locality-sensitive hashing (LSH) technique
with a novel missing-mass argument
to obtain
a fast and Bayes-consistent
classifier.
Our algorithm's prediction runtime
compares favorably against
state of the art approximate NN methods,
while maintaining Bayes-consistency and attaining rates comparable to minimax.
On samples of size $n$ in $\R^d$,
our pre-processing phase
has runtime $O(d n \log n)$,
while the evaluation phase
has runtime $O(d\log n)$ per query point.
\end{abstract}

\section{Introduction}

In the sixty or so years since the introduction of the nearest neighbor
paradigm,
a large amount of
literature has been devoted to analyzing and refining
this surprisingly
effective classification method.
Although the $1$-NN classifier is not in general Bayes-consistent
\citep{CoverHart67},
taking a majority vote among the $k$ nearest neighbors
does guarantee Bayes consistency,  provided that $k$ increases
appropriately in sample size 
\citep{stone1977,MR780746,MR877849}.
However, the $k$-NN classifier presents issues of its own.
A naive implementation involves storing the entire sample,
over which a linear-time search is performed when
evaluating the hypothesis
on test points.
For large samples sizes, this approach is prohibitively expensive in terms
of storage memory and computational runtime.

Until recently, research on NN-based methods tended to focus
somewhat dichotomously
either
on
the statistical or the computational aspects.
On the statistical front,
the most commonly investigated questions involve Bayes consistency and rates
of convergence
under various distributional assumptions
\citep{samworth2008,kpo09,gadat2016,DBLP:journals/corr/ChaudhuriD14}.

An orthogonal
body of literature
developed a host of techniques for
evaluating the hypothesis (or an approximation to it)
on
test points in runtime considerably better than linear in sample size. 
In this setting,
exact NN search methods suffer from either space or query time that is exponential in the dimension $d$
\citep{samet2006foundations}.
To overcome this problem,
{\em approximate} NN search was proposed.
Broadly speaking, these techniques construct
a hierarchical net during the offline pre-processing (learning)
phase
\citep{KL04,BKL06,DBLP:journals/tit/GottliebKK14+colt},
or seek to {\em condense} the sample down to a smaller yet
nearly-faithful subsample
\citep{DBLP:journals/tit/Hart68,gates72,ritter75,wilson00,gkn-ieee18+nips},
or perform some sort of dimensionality reduction
\citep{indyk1998approximate,DBLP:conf/stoc/Charikar02,
datar2004locality,
  DBLP:journals/cacm/AndoniI08,GottliebKK13tcs+alt}.
The speedup in search time is offset by
a degraded classification accuracy, and
with rare exceptions \citep{DBLP:journals/tit/GottliebKK14+colt},
this tradeoff has not been addressed in the literature.

The aim of this paper is to combine the best of both worlds:
to reap the advantages of fast evaluation time
while maintaining Bayes consistency, with the risk decaying at a rate
not much worse than minimax.
We combine the {\em locality-sensitive hashing} (LSH) technique
of \citet{datar2004locality} with a novel missing-mass argument
to construct a fast, Bayes-consistent LSH-based classifier.

\paragraph{Our contribution.}
Our main contribution consists of
constructing
a fast and Bayes-consistent
classifier
in $\R^d$.
Our algorithm's prediction runtime
compares favorably against
state of the art approximate NN methods.
An additional advantage our method enjoys over
the latter is provable Bayes-consistency ---
and a convergence rate that is off by a power of $2$
from the minimax rate.
The concentration inequality
for a generalized notion of missing mass
developed in the course of our analysis
may be of independent interest.

\paragraph{Related work.}
Following the pioneering work of \citet{CoverHart67},  
it was shown by \citet{MR780746,MR877849} that the k-NN classifier is strongly Bayes-consistent.
Some of the classic results on $k$-NN risk decay rates were later refined 
by taking into account the noise margin, i.e., the data distribution around
the decision boundary.
In particular,
\citet{DBLP:journals/corr/ChaudhuriD14}
obtain
minimax
rates of the form $O(n^{-\frac{\alpha(\beta+1)}{2\alpha + d}})$,
where $\alpha$ is a H\"older-like smoothness exponent
of the regression function
$\eta(x) = \mathbb{P}(Y= 1 | X=x)$
and $\beta$ is a Tsybakov noise exponent.
To obtain this rate, they require
$k = \Theta(n^{\frac{2\alpha}{2\alpha+d}})$,
which slows down the query time by an additional
$\poly(n)$ factor.
A recently proposed alternative approach, based on sample compression and
$1$-NN classification has been shown to be Bayes-consisetnt in
doubling metric spaces
\citep{DBLP:conf/nips/KontorovichSW17} --- and in fact is universally consistent in all spaces
where Bayes consistency is possible \citep{hksw19}.

Various approximate NN techniques have been proposed to speed up the query time.
One such result was obtained by
\citet{har2012approximate},
who show that
$(r,cr,p_1,p_2)$-sensitive LSH families
(see definition below)
achieve an approximate NN query time of $O(dn^\rho)$,
where $\rho=\frac{\log(1/p_1)}{\log(1/p_2)}$.
Other approximation methods include fast $\eps$-net constructions \citep{KL04},
where query time
(after sample compression, as in \citet{gkn-ieee18+nips})
is $O(1/\eps^d)$ but does not depend on $n$.
No risk convergence (or even Bayes consistency)
analysis is known for any classifier using these methods
--- absent which, as we argue in the discussion below
Table~\ref{tab:comparison}, comparisons to our approach are not meaningful.

The recent work of \citet{DBLP:conf/aistats/XueK18} proposes
aggregating denoised $1$-NN predictors over a small number of
distributed subsamples.
This approach, which requires distributed computing resources,
can achieve nearly the accuracy of $k$-NN while matching
the prediction time of $1$-NN.
Since the present paper does not assume access to parallel processors,
this result is incomparable to ours.

\paragraph{Paper outline.} The structure of this paper is as follows.
Section \ref{prelim} contains the relevant definitions and notations.
Section \ref{main-res} discusses our main contributions.
In section \ref{sec:LSH} we present the LSH based learner algorithm.
Full detailed proofs are deferred to the supplementary material.

\section{Preliminaries}
\label{prelim}
\paragraph{Learning model.}
We work in the standard agnostic learning model \citep{mohri-book2012,shwartz2014understanding}, whereby the learner receives a sample $S$ consisting of $n$ labeled examples $\{(x_i,y_i)\}_{i=1}^n$ drawn iid from an unknown distribution $\mathcal{D}$ over $\mathcal{X}\times \mathcal{Y}$. 
In this work we take $\mathcal{X}=[0,1]^d$
equipped with an $\ell_p$ metric $\nrm{x-x'}_p^p=\sum_{i=1}^d|x_i-x'_i|^p$;
when the subscript $p$ is omitted, its default value is always $p=2$: $\nrm{\cdot}\equiv\nrm{\cdot}_2$.
For simplicity of exposition,
we take
$\mathcal{Y}=\{0,1\}$;
the extension to the multiclass case is straightforward\footnote{by replacing ``majority vote'' in Section~\ref{sec:LSH} by the plurality label,
as done in \citet{DBLP:conf/nips/KontorovichSW17}}.

Let $\mathcal{D}_{\mathcal{X}}$ denote the induced marginal distribution over $\mathcal{X}$ and let $\eta$ be the conditional probability over the labels:
$\eta(x)=\mathbb{P}(Y=1 | X=x)$.
This function is said to be $(\alpha,L)$-H\"older if
\beqn
\label{eq:holder}
|\eta(x)-\eta(x')|\le L\nrm{x-x'}_p^\alpha,
\qquad x,x'\in\X
.
\eeqn
Based on the training sample $S$,  the learner produces a hypothesis $h:\mathcal{X}\to \{0,1\}$ whose empirical error is defined by $\hat R_n(h)=\frac{1}{n}\sum_{i=1}^n \pred{h(x_i)\neq y_i}$ and whose generalization error is defined by $R(h) = \mathbb{P}(h(x)\neq y)$.
The Bayes-optimal risk is defined as
$R^* = \inf_h R(h)$,
where the infimum is over all measurable hypotheses.
This infimum is achieved by the Bayes-optimal classifier, $h^*$, given by
\beq
h^*(x) = \argmax\limits_{y\in\{0,1\}} \mathbb{P}(Y=y | X=x).
\eeq
A learning algorithm mapping a sample $S$ of size $n$ to a
hypothesis $h_n$ is said to be (weakly) Bayes-consistent
if $\lim_{n \to \infty}\mathbb{E}[R(h_n)] = R^*$.
(For {\em strong} Bayes consistency, the convergence is
almost-sure rather than in expectation, but this paper deals
with the former notion.)

\paragraph{Locality Sensitive Hashing.}
Let $\mathcal{H}$ be a family of {\em hash} functions mapping  a metric space
$( \mathcal{M} , \rho )$  to some set $U$.
The family $\mathcal{H}$ is called $(r,cr,p_1,p_2)$-{\em sensitive}
if  for any two points $ x,x'\in \mathcal{M}$, using a function $h\in \mathcal{H}$  which is drawn from some distribution $\mathbb{P}_{\mathcal{H}}$: 
\begin{itemize}
\item  $\rho(x,x')\leq r \implies
  \mathbb{P}_{\mathcal{H}}\Big(h(x)=h(x')\Big) \geq p_1$,
\item $\rho(x,x')\geq cr \implies \mathbb{P}_{\mathcal{H}}\Big(h(x)=h(x')\Big) \leq p_2
  $.
\end{itemize}
In order for a locality-sensitive hash (LSH) family to be useful, it
must satisfy inequalities $p_1 > p_2$ and $c>1$ \citep{datar2004locality}.

{
\paragraph{$k$-missing mass.}
For a sample $S=(X_1,\ldots,X_n)$
drawn iid from a discrete distribution $P=(p_1,p_2,\ldots)$
over $\N$, the {\em missing mass}
is the total mass of the atoms (i.e., points in $\N$) not appearing in $S$.
Let us define a generalized notion, the $k$-missing mass.
For $i,k \in [n]$ and $j\in \N$, define $\psi_{i,j}=\mathds{1}[X_i=j]$ and $\xi_j ^{(k)} = \mathds{1}[\Sigma_{i=1}^n \psi_{i,j} < k]$; in words, $\xi_j ^{(k)}$ is the indicator for the event that the $j$th atom was observed fewer than $k$
times. The $k$-missing mass is the following random variable:
\beqn
\label{eq:Udef}
U_n ^{(k)} = \sum\limits_{j \in \N} p_j \xi_j ^{(k)}
\eeqn
(for $k = 1$, this is the usual missing mass).
}

\section{Main Results}
\label{main-res}
Our first contribution is the construction of a sequence $\mathcal{H}_n$ of $(r_n,cr_n,p_{1},p_{2})$-{\em sensitive} families with the following properties:
\begin{itemize}
\item[S1.] $p_{1}^2 > p_{2}$  
\item[S2.] $r_n \to 0$ as $n \to \infty$
\item[S3.] $\frac{1}{r_n} = o(\sqrt{n})$.
\end{itemize}
Following \citet{datar2004locality},
our construction (given in Section~\ref{sec:lsh-family})
is based on $p$-stable distributions.

Using this construction, we design a learning algorithm
(Alg.~\ref{alg:lsh})
with runtime $O(dn \log{n})$, for the pre-processing phase and evaluation (online) runtime $O(d\log{n})$.
The pre-processing phase and evaluation times are compared to other algorithms in Table~\ref{tab:comparison}.

\begin{table}
    \begin{tabular}{l|l|l} %
      \textbf{Algorithm} & \textbf{Training time} & \textbf{Query time}\\
      \hline
      $k$-NN & $O(1)$ & $O(dkn)
      $ \\
      OptiNet
      & $O(dn^4)
      $ & $O(dn)
      $
      \\
      this paper
      &  $O(dn \log{n})$ & $O(d \log{n})$ \\
    \end{tabular}
    \caption{A comparison of the various algorithms' runtimes
     (OptiNet is given in Algorithm 1 of \citet{hksw19}).
      Note that while the query time of $k$-NN may be improved (e.g., to $O(dkn^{1/c^2})$
      using an LSH family) and the training time of
      OptiNet can be improved to $O(C_{d,\eps}n\log n)$ via fast $\eps$-net \citep{DBLP:journals/tit/GottliebKK14+colt},
      the effect of the approximate NN techniques on Bayes consistency is not understood --- much less the effect on
      the risk decay rates. Indeed, one can trivially speed up any learning algorithm by discarding all but
      a tiny fraction of the training sample. This will obviously significantly degrade the risk rate,
      which underscores that runtime comparisons are only meaningful among techniques with comparable risk rates.
    }
    \label{tab:comparison}  
\end{table}

In addition to achieving an exponential speed-up over the state of the art, our algorithm enjoys the property of being Bayes-consistent. The price we pay for the computational speedup is a quadratic slow-down of the convergence rate:
\begin{theorem}
\label{thm:const}
Let $\mathcal{X}=[0,1]^d, \mathcal{Y}=\{0,1\}$, and $\mathcal{D}$
be a distribution over $\mathcal{X}\times \mathcal{Y}$
for which the conditional probability function, $\eta$, is
$(\alpha,L)$-H\"older. Let $f_{n}$
denote the
classifier constructed by Algorithm~\ref{alg:lsh})
on
a sample $S_n \sim \mathcal{D}^n$.
Then
the LSH learner is weakly Bayes-consistent:
$\lim_{n\to\infty}\mathbb{E}[R(f_{n})] = R^* .$
Further,
$
\mathbb{E}[R(f_{n})] - R^*
=
O(n^{-\frac{\alpha}{2d+6}})$.
\end{theorem}
\paragraph{Remark.}
Since we rely on the LSH techniques developed by \citet{indyk1998approximate,datar2004locality,DBLP:journals/cacm/AndoniI08},
it might appear that we are
``beating them at their own game'' by achieving an exponential speedup over the state-of-the-art runtimes based on LSH.
A more accurate conceptual explanation would be that we are ``playing a different  game''. Namely,
while the latter works
focus on the approximate nearest neighbor problem, our goal is rather to efficiently label a test point,
without guaranteeing anything about its approximate nearest neighbor in the sample.
Instead, we guarantee that with high probability,
most of the points
in
a query point's
hash bucket will be
in its close proximity.

\paragraph{Open problem.} Is
there
an NN-based classification
algorithm with query evaluation time $O(d\log{n})$ that achieves,
under the conditions of
Theorem~\ref{thm:const},
the minimax risk rate
of
$O(n^{-\frac{\alpha}{2\alpha + d}})$?

Our analysis is facilitated by
a bound on the $k$-missing mass of possible independent interest:
\begin{theorem}
\label{thm:mss}
Let $U_n^{(k)}$ be the missing mass variable defined in  (\ref{eq:Udef}). For $\eps>0$, $n\in\N$ and $1\leq k \leq n$, we have
\begin{itemize}
  \item[(a)]
$\mathbb{E}[U_n^{(k)}] < 1.6{\nrm{P}_0k}/{n}$,
    where\\ $\nrm{P}_0 = \sum_{j\in \N} \mathds{1}[p_j >0]$ is the support size of $P$;
  \item[(b)]
$\mathbb{P}\big(U_n^{(k)} >\mathbb{E}[U_n^{(k)}] +\eps \big) \leq 2\exp\paren{-0.09{n\eps^2}/{k}}$.
\end{itemize}
\end{theorem}
\paragraph{Remark.}
Lemma 16.6 in \citet{shwartz2014understanding} claims the bound
$\mathbb{E}[U_n^{(k)}] \le 2{\nrm{P}_0k}/{n}$ for $k\ge2$. The proof is an
exercise, but a sketch is provided. Since we provide a complete proof (via a
different method), with a better constant and without restricting the range of $k$,
we decided to include part (a) above.
The concentration result in (b) is, to our knowledge, novel.

\section{LSH based Learner}
\label{sec:LSH}
Our LSH-based algorithm (presented formally in Alg.~\ref{alg:lsh}) operates as follows. Given a sample $S_n$ of size $n$,
we set the radius parameter $r_n$, and pick $m_n = O(\log{n})$ functions
$\set{h_i}$
from an LSH family $\mathcal{H}_n$, and define $g_n(x)=(h_1(x),\ldots,h_{m_n}(x))$.
Using $g_n$ we then
we construct the hash table $T$,
which contains the training set $S_n$,
and each bucket is labeled according to the majority vote among the labels of the $x_i$'s falling into the bucket. Technically, this is done by taking a single pair, which agrees with the
majority vote, $(x_i, y_i)$, from the bucket, and inserting it into
a new table $T'$, using the same hash function $g_n$.
The LSH learner runs in $O(dn \log{n})$, and its output is
a
classifier defined by a
(table, hash function) pair.

We denote by $|T|$ the size of the table, namely, the number of buckets in $T$. We use $|T(k)|$ to denote the number of elements in the bucket whose key is $k$. The number of buckets can be reduced, by retaining only the nonempty buckets using (standard) hashing of the values $g_n(x)$. However, in this work we use single hashing.

\begin{algorithm}[ht]
  \caption{LSH based
    learner}\label{alg:lsh}
  \begin{algorithmic}[1]
    \Require
      \Statex Sample $S_n = \{(x_i,y_i)\}_{i=1}^n$
    \Ensure
      \Statex LSH based classifier 
    \Statex
    \State set $m_n = \lfloor\frac{\log{n}}{2\log{\frac{1}{p_1}}}\rfloor$
    \State pick $m_n$ functions from $\mathcal{H}_n$ where $\mathcal{H}_n$ is as in Section \ref{sec:lsh-family}
  	\State Initialize empty hash tables $T,T'$
    \State set $g_n=(h_1,\ldots,h_{m_n})$
    \For{$i=1 \to n$}
   		\State add $(x_i,y_i)$ to $T(g_n(x_i))$
    \EndFor
    \For{bucket $j$ in $T$}
   		\If{$\sum\limits_{(x_i,y_i)\in T(j)}y_i > \frac{|T(j)|}{2}$}
   			\State find $(x',y')\in T(j)$ s.t. $y' = 1$
   			\State add $(x',y')$ to $T'(g_n(x'))$
   		\Else
   			\State find $(x',y')\in T(j)$ s.t. $y' = 0$
   			\State add $(x',y')$ to $T'(g_n(x'))$
   		\EndIf
                \EndFor
	\State return $(T',g_n)$
  \end{algorithmic}
\end{algorithm}

To label 
a test point $x$, we need to access the label in $T'(g_n(x))$.
This can be done in time $O(d \log{n})$ (see Algorithm \ref{alg:algorithm2}).

\begin{algorithm}[ht]
\caption{LSH based classifier $f_{T',g_n}$}\label{alg:algorithm2}
  \begin{algorithmic}[1]
    \Require
    \Statex hash table $T'$
    \Statex hash function $g_n$
     \Statex test point $x\in\mathcal{X}$
	\If{$T'(g_n(x))$ is not empty}
    	\State $(x',y') \gets$ retrieve element from $T'(g_n(x))$
    	\State return $y'$
    \Else
    	\State return default label $0$   
    \EndIf
  \end{algorithmic}
\end{algorithm}

\subsection{LSH familiy} \label{sec:lsh-family}
The term Locality-Sensitive Hashing (LSH) was introduced by \citet{indyk1998approximate}
to describe
a randomized hashing framework for efficient approximate nearest neighbor
search in high-dimensional space. It is based on the definition of LSH
family $\mathcal{H}$, a family of hash
functions mapping similar input items to the same hash code with higher probability than dissimilar items. Our LSH learner is using the following family, proposed by
\citet{datar2004locality}.
For the Euclidean metric we pick a random projection of $\mathbb{R}^d$ onto a 1-dimensional line and chop the line into segments of length $w$, shifted by a random value $b \in [0, w)$. Formally, $h_{\alpha,b}(x) = \lfloor \frac{\alpha x + b}{w} \rfloor$, where the projection vector $\alpha \in \mathbb{R}^d$ is constructed by picking
  each coordinate of $\alpha$ from the
standard normal $N(0,1)$
  distribution.
  The choice of $w$ is made according to the sample size. A generalization of this approach to $\ell_p$ norms for any $p\in(0, 2]$ is possible as well; this is done by picking the vector $\alpha$ from so-called $p$-stable distribution. We compute the probability that two vectors $v_1, v_2 \in \mathbb{R}^d$ collide under a hash function
drawn from this family. For the two vectors, let $z = \nrm{v_1-v_2}_p$ and let $P(z)$ denote the probability
that $v_1, v_2$ collide for a hash function chosen from the family $\mathcal{H}$ described above. For a random vector $\alpha$ whose entries are drawn from a $p$-stable distribution, $\alpha v_1 - \alpha v_2$ is distributed as $zX$ where $X$ is a random variable drawn from a $p$-stable distribution. We get a collision if both $|\alpha v_1 - \alpha v_2|< w$ 
and a divider does not fall between $\alpha v_1$ and $\alpha v_2$. It is easy to see that
\beq
\mathbb{P}(h(v_1)=h(v_2))=P(z)=\int_{0}^\frac{w}{z}  \phi_p(t)(1-\frac{tz}{w})dt ,
\eeq
where $\phi_p$ is the density of the absolute value of the $p$-stable distribution. Notice that for a fixed $w$, this probability depends only on the distance $z$, and it is monotonically decreasing in $z$. Finally, given a sample $S$ of size $n$, we set 
\beq
w = \Big(\frac{1.6{d}^{(d+2)/2}}{n^{\frac{d+1}{2d+6}}}\Big)^{\frac{1}{d+1}}
.
\eeq
Choosing $r_n=w$, we get 
\begin{align*}
&p_{1}=P(r_n)=\int_{0}^1  f(t)(1-t)dt, \\
&p_{2}=P(cr_n)=\int_{0}^\frac{1}{c}  f(t)(1-ct)dt.
\end{align*}
For example, for the Euclidean norm,
we have
$\phi_p(t)=\frac{2}{\sqrt{2\pi}}e^{- \frac{t^2}{2}}$ and $c= 3$,
which induces
a $(r_n, 3r_n, p_{1},p_{2})$-sensitive family with
\begin{align*}
&p_1 = P(r_n) \approx 0.367691,\\
&p_2 = P(3r_n) \approx 0.131758. 
\end{align*}

More generally, our Bayes consistency results hold for the LSH learner
whenever the $(r_n,cr_n,p_{1},p_{2})$-sensitive family $\mathcal{H}_n$
satisfies the properties S1-S3.

\section{Proof of Theorem \ref{thm:mss}(a)}

\paragraph{Remark.}
As shown in 
\citet{Berend20121102},
even for $k=1$,
one cannot, in general, obtain estimates on
$\E[U_n^{(k)}]$ independent of the support size
--- unlike concentration bounds,
which are dimension-free.
\begin{proof}
Decompose
$U_n^{(k)}=X+Y$, where
\beqn
\label{eq:XY}
X = \sum_{j: k\le np_j}p_j\xi_j^{(k)},
\qquad
Y = \sum_{j: k> np_j}p_j\xi_j^{(k)}.
\eeqn
Then $\E[U_n^{(k)}]=\E[X]+\E[Y]$ and
\beq
\E[\xi_j^{(k)}] = \P(\mathrm{Bin}(n,p_j)<k) = \sum_{\ell=0}^{k-1}{n\choose\ell}p_j^\ell(1-p_j)^{n-\ell}.
\eeq
For $k\le np_j$,
the multiplicative Chernoff bound
$\P(\mathrm{Bin}(n,p)<(1-\delta)np)\le\exp(-\delta^2np/2)$
yields
$\E[\xi_j^{(k)}] \le\exp\paren{-\frac{(np_j-k)^2}{2np_j}},$
whence
\beqn
\label{eq:EX}
\E[X] \le \sum_{j: k\le np_j}p_j\exp\paren{-\frac{(np_j-k)^2}{2np_j}}.
\eeqn
We estimate this quantity via the simple strategy
of
maximizing each summand independently over $p_j$.
To this end,
define the function
$F(p)=p\exp\paren{-\frac{(np-k)^2}{2np}}$
over
$p\in[k/n,1]$
and compute
\beq
F'(p) = \frac{\exp\paren{-\frac{(np-k)^2}{2np}}(k^2+np(2-np))}{2np}.
\eeq
The latter vanishes at
\beq
p\in\set{p_+,p_-}:=
\frac{1\pm\sqrt{1+k^2}}{n}
,
\eeq
of which only $p_+$ lies in the permitted range $[k/n,1]$.
Since for $k\le n$ we always have $k^2<n(n+2)$,
it follows that $F'(1)<0$,
and hence either $p_+\le1$ maximizes $F$ over $[k/n,1]$
or else $p_+>1$ (which happens iff $k^2>n(n-2)$) and
$F$ is maximized at $p=1$.
We shall analyze both cases.
For the first case,
it is a simple exercise to show that
\beq
\frac{(np_+-k)^2}{2np_+}=
\frac{(1+\sqrt{k^2+1}-k)^2}{2(1+\sqrt{k^2+1})}
\ge
\frac1{(1+\sqrt2)k}
\eeq
and hence
\beq
\frac{
nF(p_+)
}{k}
\le
\frac{
(1+\sqrt{k^2+1})
  \exp(-[(1+\sqrt2)k]\inv)
  }{k}
=:G(k).
\eeq
We claim that $G$ is monotonically decreasing in $k$.
Indeed,
$
k^3\sqrt{k^2+1}
e^{\frac{
      \sqrt2-1
    }{k}
    }
[\sqrt2-1]\inv
G'(k)
=$
\beq
k^2+1+\sqrt{k^2+1}
  -
  (\sqrt2+1)k(1+\sqrt{k^2+1})
  <0,
  \eeq
  which follows readily from
  $k\le \sqrt{k^2+1}\le k+\sqrt2-1$,
  for $k\ge1$.
  Thus,
\beq
G(k)\le G(1)=
(1+\sqrt2)\exp(-[1+\sqrt2]\inv])
<1.595457,
\eeq
whence 
\beqn
\label{eq:p+}
F(p_+) < 1.6k/n.
\eeqn
For the second case,
which requires bounding $F(1)$,
we claim
that
\beqn
\label{eq:p=1}
\sup_{n\ge1}\sup_{k\in[1,n]}
\exp\paren{-\frac{(n-k)^2}{2n}}
< 1.56k/n.
\eeqn
Indeed, putting $x=k/n$, we can define
$G(x)=\exp\paren{-\frac{n^2(1-x)^2}{2n}}/x$
and verify that $G'(x)<0$ on $[1/n,1]$.
Thus, the extreme value
of $\exp(-1/4)/2\approx1.56$
in (\ref{eq:p=1})
is achieved at
$n=2$ and $k=1$.

It follows
from (\ref{eq:p+}) and (\ref{eq:p=1})
that
\beq
\E[X] \le 1.6k/n\cdot\abs{\set{j\in\N: p_j\ge k/n}}.
\eeq

The upper bound on $\E[Y]$ is trivial:
\beq
\E[Y] = \sum_{j: k> np_j}p_j\E[\xi_j^{(k)}]
\le k/n\cdot\abs{\set{j\in\N:
    p_j>0
}}.
\eeq
Combining the estimates on $\E[X]$ and $\E[Y]$
concludes the proof.
\end{proof}

\section{Proof of Theorem \ref{thm:mss}(b)}

We begin by observing that the random variables
$\xi_j^{(k)}$, though not independent, are
{\em negatively associated}, as shown in \citet{DBLP:journals/jmlr/McAllesterO03}.
Thus, for the purpose of establishing concentration,
one may invoke the standard
Bernstein-Chernoff
exponential bounding argument verbatim
\citep{Dubhashi:1998:BBS:299633.299634}. We shall do so in the sequel
without further comment.

We maintain the decomposition
$U_n^{(k)}=X+Y$ as in (\ref{eq:XY})
and derive concentration bounds on
$X$ and $Y$ separarately.
A bound for
$U_n^{(k)}$ will then follow via
\begin{align}
\nonumber
&\P(U_n^{(k)}\ge\E[U_n^{(k)}]+\eps) 
\le \\
&
\label{eq:UXY}
\P(X\ge\E[X]+\alpha\eps)
+
\P(Y\ge\E[Y]+(1-\alpha)\eps),
\end{align}
for any choice of $0\le\alpha\le1$.

\subsubsection*{Tail bounds for $X$}
In this section, we always assume that
$n\ge1$, $p\in[0,1]$ and $1\le k\le np$.
Define the function
$q=q(k,n,p):=\exp\paren{-\frac{(np-k)^2}{2np}}$
and the collection of independent
Bernoulli variables $\xi'_j\sim\mathrm{Ber}(q(k,n,p_j))$,
as well as $X':=\sum_{j: k\le np_j}p_j \xi'_j$.
It follows
from (\ref{eq:EX}) that
$\E[X]\le\E[X']=
\sum_{j: k\le np_j}p_j q(k,n,p_j)$
and from
negative association that
\beqn
\label{eq:XX'}
\P(X\ge\E[X]+\eps)\le\P(X'\ge\E[X']+\eps),\qquad\eps>0.
\eeqn

Our strategy
for bounding (\ref{eq:XX'})
is to bound the moment generating function
$\E\exp[\lambda(X'-\E[X'])]$ --- to which end,
it suffices to bound
\beqn
\nonumber
\E\e^{\lambda p_j(
\xi_j'
  -\E[
\xi_j'
  ])}
&=&
 q(k,n,p_j)\e^{\lambda p_j(1-q(k,n,p_j))}
 \\\nonumber &+&(1-q)\e^{-\lambda p_jq(k,n,p_j)}
\label{eq:phidef}
\\ & =:&\Phi(\lambda,k,n,p_j).
\eeqn
\begin{lemma}
  \label{lem:mgf}
  For $\Phi$ as defined in (\ref{eq:phidef}),
  \beq
  \Phi(\lambda,k,n,p)
  \le
\exp(C_\Phi\lambda^2pk/n),
\eeq
where $C_\Phi
\le
(2+\sqrt3)/4\log(\e-1)<1.73$
is a universal constant.\footnote{
Numerical simulations suggest that $C_\Phi<0.61$.
  }
\end{lemma}
Armed with Lemma~\ref{lem:mgf},
the standard argument yields an estimate on 
(\ref{eq:XX'}):
\beq
\P(X'\ge\E[X']+\eps) &=&
\P(\exp(\lambda(X'-\E[X']))\ge\e^{\lambda\eps})\\
&\le&
\e^{-\lambda\eps}
\prod_{j: k\le np_j}
\E\e^{\lambda p_j(
\xi_j'
  -\E[
\xi_j'
  ])}
\\&=&
\e^{-\lambda\eps}
\prod_{j: k\le np_j}
\Phi(\lambda,k,n,p_j)
\\&\le&
\e^{-\lambda\eps}
\prod_{j: k\le np_j}
\exp(C_\Phi\lambda^2p_jk/n)
\\&\le&
\exp(C_\Phi\lambda^2k/n-\lambda\eps).
\eeq
Choosing $\lambda=\eps n/2kC_\Phi$ yields
\beqn
\label{eq:Xdev}
\P(X\ge\E[X]+\eps)\le\exp(-\eps^2n/4kC_\Phi).
\eeqn

\subsubsection*{Tail bounds for $Y$}
As done for $X$ in (\ref{eq:XX'}),
we invoke
negative association to obtain
\beqn
\label{eq:YY'}
\P(Y\ge\E[Y]+\eps)\le\P(Y'\ge\E[Y']+\eps),\qquad\eps>0,
\eeqn
where $Y'=
\sum_{j: k> np_j}p_j\xi_j'$
and the $\xi_j'\sim\mathrm{Ber}(q_j)$ are independent,
and
$q_j:=
\sum_{\ell=0}^{k-1}{n\choose\ell}p_j^\ell(1-p_j)^{n-\ell}
$.
In particular, $\E[Y]=\E[Y']$.

An application of Hoeffding's inequality yields
\beq
\P(Y'\ge\E[Y']+\eps)\le\exp\paren{-\frac{2\eps^2}
  {\sum_{j: k> np_j}p_j^2}};
\eeq
it remains to bound
$\sum_{j: k> np_j}p_j^2$.
To this end, we invoke H\"older's inequality:
$\nrm{x}_2^2\le \nrm{x}_\infty\nrm{x}_1$,
whence
\beqn
\sum_{j: k> np_j}p_j^2 \le
\frac{k}{n}.
\eeqn
It follows that
\beqn
\label{eq:Ydev}
\P(Y\ge\E[Y]+\eps) < \exp(-2\eps^2n/k).
\eeqn

From (\ref{eq:Xdev}), we have that
$
\P(X\ge\E[X]+\alpha\eps)\le
\exp(-\alpha^2\eps^2n/4kC_\Phi)
$
and from
(\ref{eq:Ydev}), that
$
\P(Y\ge\E[Y]+(1-\alpha)\eps) <
\exp(-2(1-\alpha)^2\eps^2n/k)
$.
The choice $\alpha=1/(1+(2\sqrt{2C_\Phi})\inv)\approx0.7878$
makes the two exponents equal:
\beq
\max\set{
\P(X\ge\E[X]+\alpha\eps)
,
\P(Y\ge\E[Y]+(1-\alpha)\eps)
}
\\< 
\exp(-0.09\eps^2k/n).
\eeq
Combining these with (\ref{eq:UXY})
concludes the proof.
\qed

\begin{proof}[Proof
    of Lemma~\ref{lem:mgf}]
Throughout the proof,
$n\ge1$, $p\in[0,1]$, $1\le k\le np$
and $q=q(k,n,p)$ as defined above.

As in the proof of Lemma 3.5(a) in \citet{ECP2359},
we invoke the Kearns-Saul inequality to obtain
\beq
q\exp(\lambda(p-pq))+(1-q)\exp(-\lambda pq) \le 
\\\exp[(1-2q)\lambda^2p^2/4\log[(1-q)/q]].
\eeq
Thus, to prove
the Lemma,
it suffices to show that
\beq
(1-2q)/\log[(1-q)/q]
\le
4C_\Phi k/np.
\eeq
Holding $\mu:=np$ fixed, put $x=k/\mu$ and reparametrize $q$ as $ y(x)=\exp(-\mu(x-1)^2/2)$;
our task is now reduced to proving
\beqn
\label{eq:repar}
\sup_{\mu\ge1}
\sup_{x\in[1/\mu,1]}
\frac{1-2 y(x)}{x\log[(1- y(x))/ y(x)]}
\le 4C_\Phi.
\eeqn
Note that $x\ge1/\mu$ implies $ y\ge\exp(-(\mu-1)^2/2\mu)$.
Reparametrize again via $ z:=\log(1/ y)
\le
(\mu-1)^2/2\mu
$;
now proving (\ref{eq:repar}) amounts to showing that
\beq
F( z) &:=& \frac{
  1-2\e^{- z}
}{
  \paren{ 1-\sqrt{2z/\mu}}
  \log(\e^{z}-1)}
\le 4C_\Phi,\\
&&
\text{for }
\mu\ge1,  z\in[0,(\mu-1)^2/{2\mu}].
\eeq
Our proof will not require this, but we note that $F$ is always non-negative;
this is clear from the parametrization in (\ref{eq:repar}).
To prove (\ref{eq:repar}), we consider the two cases $z<1$ and $ z\ge1$
below,
from which the estimate $C_\Phi\le
(2+\sqrt3)/4\log(\e-1)<1.73$
readily follows.

\paragraph{Case I: $ z<1$.}
This case will follow from the inequalities
\beqn
\label{eq:A/C}
\sup_{0<z<1}
\abs{
\frac{1-2\e^{-z}}
{\log(\e^{z}-1)}
}
\le\frac12
\eeqn
and
\beqn
\nonumber
\sup_{\mu\ge1}
\quad
\sup_{0<z<\min\set{1,\frac{(\mu-1)^2}{{2\mu}}}}
\abs{
  \frac{1}
       {
  1-\sqrt{2z/\mu}
       }
}\\\nonumber
\label{eq:1/B}
\\\le 2+\sqrt3\approx3.73;
\eeqn
combining them implies a bound of $F(z)\le1+\sqrt{3}/2\approx1.87$
over the specified range of $\mu$ and $z$.
Both (\ref{eq:A/C}) and 
(\ref{eq:1/B}) are straightforward exercises.
The former is facilitated by the reparametrization $(1-2/t)/\log(t-1)$
while the latter involves analyzing the two cases 
$(\mu-1)^2/{2\mu}\gtrless1$,
whose boundary is demarcated by
$\mu=2+\sqrt3$.

\paragraph{Case II: $z\ge1$.}
This case
is facilitated by
the trivial estimate
\beqn
\label{eq:t^2}
\sup_{t\ge1} \frac{t}{\log(\e^{t}-1)} &\le&
1/\log(\e-1)<
1.85.
\eeqn
Indeed, since
$|1-{2\e^{-z}}|\le1$, it follows from (\ref{eq:t^2})
that
\beq
F(z)\le
G(z):=\frac{1.85}{
z(1-\sqrt{2z/\mu})}
\eeq
over the specified range of $\mu$ and $z$, which is
$z\in[1,(\mu-1)^2/{2\mu}]$
and
$\mu\ge2+\sqrt{3}
$ (since for smaller $\mu$, the range of $z$ is empty).
Now the function $G(z,\mu):= z(1-\sqrt{2z/\mu})$ is unimodal in $z$ for
fixed $\mu$,
vanishing at $z=0$ and at $z=\mu/2$, and achieving a
positive
maximum
value
strictly inbetween. As the actual range of $z$ is
$1\le z\le (\mu-1)^2/2\mu
<\mu/2$ (the latter inequality holds for all
$\mu\ge1$),
to analyze the minimum of $G(\cdot,\mu)$,
we need only consider
the extreme feasible values
$z_1=1$ and $z_2= (\mu-1)^2/2\mu$.
A straightforward computation yields
\beq
\sup_{\mu\ge2+\sqrt{3}}\frac1{G(z_1,\mu)}
&=&
\sup_{\mu\ge2+\sqrt{3}}\frac1{1-\sqrt{2/\mu}}\\
&=&
\frac1{1-\sqrt{4-2\sqrt{3}}}\\
&=&2+\sqrt3
\eeq
and
\beq
\sup_{\mu\ge2+\sqrt{3}}\frac1{G(z_2,\mu)}
=
\sup_{\mu\ge2+\sqrt{3}}
\frac
{2\mu^2}
{(\mu-1)^2}
=2+\sqrt3.
\eeq
Combining these implies a bound of $F(z)\le 
(2+\sqrt3)/\log(\e-1)<6.9$
over the specified range of $\mu$ and $z$.
\end{proof}

\section{Proof of Theorem \ref{thm:const}}
Our proof closely follows the argument in \citet[Theorem 6.1]{MR1383093}.

Given a test point $x\in [0,1]^d$ drawn from $\mathcal{D}_{\mathcal{X}}$, and $g_n(x)=j$, We would like to know how many sample points are in the bucket $T(j)$, and what is the ratio of the near (i.e. at distance at most $<cr_n$) and distant (i.e. at distance at least $\geq cr_n$) points in the bucket. To deal with these questions, we first set some notations.
Given a test point $x \sim \mathcal{D}_{\mathcal{X}}$ and a hash function $g_n$, we denote by $A(x)$ the set of points from $S$ in the same bucket with $x$, and $N(x)$ is the size of that bucket. Formally, 
\begin{align*}
&A(x) =  \{x_i \in S_n | g_n(x_i)=g_n(x) \}  \\
&N(x)=\Sigma_{i=1}^{n}{\pred{x_i \in A(x)}}.
\end{align*}
In addition, for $r>0$ we denote by $\Aclose(x)$  the  set of near points from $S$ in the same bucket with $x$, 
\begin{align*}
\Aclose(x) = \{x_i \in S_n | g_n(x_i)=g_n(x), \nrm{x-x_i} < cr_n \} 
\end{align*}
and $\Afar(x)$ is the complementary $A(x) \setminus \Aclose(x)$. 
Finally, we define $\Nclose(x)$ and $\Nfar(x)$ as the cardinality of the sets $\Aclose(x)$ and $\Afar(x)$. Equipped with the preceding notations, we are now ready to prove the
Theorem \ref{thm:const}.

Define $\hat\eta_n(x) = \frac{1}{N(x)} \Sigma_{i:x_i \in A(x) }y_i$ and $\eta^* (x) = \mathbb{E}\big[\eta(x') \big| x' \in \Aclose(x)\big] $. By
\citet[Theorem 2.2]{MR1383093},
we have
\beq
\mathbb{E}[R(f_{g_n,T'})] - R^* \leq 2\mathbb{E}\big[|\hat\eta_n(x)-\eta(x) |\big].
\eeq
 By the triangle inequality,
\begin{align*}
\mathbb{E}\big[|\hat\eta_n(x)-\eta(x) |\big] \leq & 
\mathbb{E}\big[|\hat\eta_n(x)-\eta^*(x) |\big]\\ + & \mathbb{E}\big[|\eta^*(x) - \eta(x) |\big].
\end{align*}
By conditioning on the variables $\pred{x_i\in A(x)}$, $\pred{x_i \in \Aclose(x)}$, it is easy to see that $\Sigma_{i:x_i \in \Aclose(x) }y_i$ is distributed as $\mathrm{Bin}(\Nclose(x),\eta^*(x))$, a binomial random variable with parameters $\Nclose(x),\eta^*(x)$. Thus,
\begin{align*}
\mathbb{E}\Big[&|\hat\eta_n(x)-\eta^*(x)| \Hquad \Big|  \Hquad \pred{x_i\in A(x)}, \pred{x_i \in \Aclose(x)} \Big] \\ & \leq
\mathbb{E}\Big[|\frac{1}{N(x)} \Sigma_{i:x_i \in A(x) }y_i -\eta^*(x) | \Hquad \Big| \\& \Hquad \pred{x_i\in A(x)}, \pred{x_i \in \Aclose(x)} \Big] +\pred{N(x)=0} 
\\ & \leq \mathbb{E}\Big[|\frac{1}{N(x)} \Sigma_{i:x_i \in \Aclose(x) }y_i -\eta^*(x) | \Hquad \Big| \\& \Hquad \pred{x_i\in A(x)}, \pred{x_i \in \Aclose(x)} \Big] + \frac{\Nfar(x)}{N(x)}  +\\& \Hquad \pred{N(x)=0} \\&=\mathbb{E}\Big[|\frac{\mathrm{Bin}(\Nclose(x),\eta^*(x))-N(x)\eta^*(x)}{N(x)} | \Hquad \Big| \\& \Hquad \pred{x_i\in A(x)}, \pred{x_i \in \Aclose(x)} \Big] + \frac{\Nfar(x)}{N(x)}  +\\& \Hquad \pred{N(x)=0} = (*)+(**)+(***).
\end{align*}
By Cauchy-Schwarz we have
\begin{align*}
(*) &\leq \Big( \frac{1}{N(x)^2}\mathbb{E}\big[(\mathrm{Bin}(\Nclose(x),\eta^*(x))-N(x)\eta^*(x))^2\big] \Big)^{\frac{1}{2}} \\  &= \Big(\frac{1}{N(x)^2}\big(\mathbb{E}\big[\mathrm{Bin}(\Nclose(x),\eta^*(x))^2\big]- \\ &\quad\quad 2\Nclose(x)N(x)\eta^*(x)^2 +N(x)^2\eta^*(x)^2\big)\Big)^{\frac{1}{2}} \\& =  \Big( \frac{1}{N(x)^2}\big(\Nclose(x)\eta^*(x)(1-\eta^*(x)) \\&\quad\quad + \eta^*(x)^2(N(x)-\Nclose(x))^2) \big) \Big)^{\frac{1}{2}}.
\end{align*}
Hence,
\begin{align*}
(*) &\leq \sqrt{\frac{\Nclose(x)}{4N(x)^2} + \Big(\frac{\Nfar(x)}{N(x)}\Big)^2}
\\& \leq \sqrt{\frac{1}{4N(x)} + \Big(\frac{\Nfar(x)}{N(x)}\Big)^2}.
\end{align*}
Hence,
\begin{align*}
  \mathbb{E}\Big[&|\hat\eta_n(x)-\eta^*(x)| \Hquad \Big|  \Hquad \pred{x_i\in A(x)}, \pred{x_i \in \Aclose(x)} \Big] \\&\leq  \sqrt{\frac{1}{4N(x)} + \frac{(\Nclose(x)-N(x))^2}{N(x)^2}} +\frac{\Nfar(x)}{N(x)}
\\&
  +\pred{N(x)=0}.
\end{align*}

Taking expectations,
\begin{align*}
\mathbb{E}\big[|\hat\eta_n(x)-\eta^*(x) & |\big]  \leq
\\& \mathbb{E}\Big [\sqrt{\frac{1}{4N(x)} + \frac{\Nfar(x)^2}{N(x)^2}} + \frac{\Nfar(x)}{N(x)} \Big] \\& \quad + \mathbb{P}(N(x)=0)
\\& \leq  \Big(\sqrt{2}+2\Big)\Big(\mathbb{P}(N(x)<M) \\& \quad + \mathbb{P}(\Nfar(x)> \delta N(x))\Big) \\& \quad + \sqrt{\frac{1}{4M} + \delta^2} + \delta.
\end{align*}
For the second term,  $\mathbb{E}\big[|\eta^*(x) - \eta(x) |\big]$ we use the smoothness assumption on $\eta$.
Since $\eta^* (x) = \mathbb{E}\big[\eta(x')  \big|  \nrm{x-x'} \leq cr \big] $ then 
\beq
\eta(x)-L(cr_n)^{\alpha} \geq \eta^*(x) \leq \eta(x)+L(cr_n)^{\alpha} .
\eeq
Hence, 
\beq
\mathbb{E}\big[|\eta^*(x) - \eta(x) |\big]\leq L(cr_n)^{\alpha} .
\eeq
Now, by applying Lemmas \ref{lem:hash}, \ref{lem:hash2}, and setting
 $\delta=\sqrt{\frac{1}{n^{s-\frac{1}{2}}}},$ $M=\frac{n^{s-\frac{1}{2}}}{4}$ we get
\begin{align*}
\mathbb{E}[R(f_{T',g_n})]& - R^*  \leq
\\& 4\Bigg(
2\exp(-\frac{1}{8}n^{s-\frac{1}{2}}) +4\exp(-0.09n^{\frac{1-s}{2}})\\&+
2\Big(\frac{1.6}{n^{1-s}\sqrt{d}^d}\Big)^{\frac{1}{d+1}} + \sqrt{\frac{4}{n^{\frac{1-s}{2}}}}+ 2^{-\frac{n^{\frac{2s-1}{4}}}{2}}
 \Bigg) \\&+ \sqrt{\frac{1}{4M} + \delta^2} + \delta +  L(cr_n)^{\alpha} .
\end{align*} 
Finally, we set $s=\frac{d+5}{2d+6}$, and for $d\geq 3$, we get
by straightforward calculation, 
\begin{align*}
\mathbb{E}[R(f_{T',g_n})] - R^* & \leq\\
&48\exp(-0.09n^{\frac{1}{2d+6}}) +  \frac{73Lc^{\alpha}\sqrt{d}^{\frac{d+2}{d+1}}}{n^{\frac{\alpha}{2d+6}}},
\end{align*}
which completes the proof.

\appendix

\section{Proof of Lemma \ref{lem:hash}}
The following lemma states that the with high probability, the ratio $\frac{\Nfar(x)}{N(x)} \to 0$ as $n$ approaches  $\infty$.
Throughout this section, $B(x,r)$ denotes the closed Euclidean $r$-ball about $x$.

\begin{lemma}
  \label{lem:hash}
Let $x \sim \mathcal{D}_{\mathcal{X}}$. Then, for all $\delta>0$, $\frac{1}{2}<s <1$, the hash table $T$ calculated by Algorithm 1 satisfies:
\begin{align*}
  \mathbb{P}\Big(\Nfar(x)>& \delta N(x)\Big) \leq
  2\exp(-0.09n^{\frac{1-s}{2}}) +\Big(\frac{1.6}{n^{1-s}\sqrt{d}^d}\Big)^{\frac{1}{d+1}}
  + \sqrt{\frac{1}{n^{\frac{1-s}{2}}}} + \exp(-\frac{1}{8}n^{s-\frac{1}{2}})
+ 2^{-\frac{\delta}{2}n^{s-\frac{1}{2}}},
\end{align*}
where the probability is over $S_n,x \sim \mathcal{D}^{n+1}$ and the choice of the function $g_n$.
\end{lemma}

\begin{proof}
 Fix $\delta>0$, $\eps = \frac{r_n}{\sqrt{d}}$, 
 and let  $C_1,\ldots,C_t$ be a
partition
of $[0,1]^d$
into $t=(\frac{1}{\eps})^d$
boxes of length $\eps$.
Notice that for any $x,x'$ in the same box, we have $\nrm{x-x'}\leq \sqrt{d}\eps$.
Put $k=n^{s}$
and
define
the random variable $L_{\eps,k}(S_n)=\Sigma_{i:|C_i \cap S_n|<k}\P(C_i)$,
and note that it is precisely the $k$-missing mass
(defined in (\ref{eq:Udef}))
associated with the distribution $P=(\P(C_1),\ldots,\P(C_t))$.
By Theorem~\ref{thm:mss}(a),
we have $\mathbb{E}[L_{\eps,k}(S_n)]\leq \frac{1.6kt}{n}$.
By the law of total probability,
\begin{align}
\mathbb{P}\Big(\Nfar(x)> & \delta N(x)\Big)
\leq \mathbb{P}\Big(L_{\eps,m}(S_n)>\frac{1.6}{\eps^d n^{1-s}} +\gamma \Big)
+ \mathbb{P}\Big(\Nfar(x)> \delta N(x)
\Big|
L_{\eps,m}(S_n) \leq  \frac{1.6}{\eps^d n^{1-s}} + \gamma \Big).
\label{dec1}
\end{align}
For the first term in (\ref{dec1}), we apply Theorem~\ref{thm:mss}(b):
\begin{align*}
  \mathbb{P}\Big(L_{\eps,m}(S_n)> & \frac{1.6}{\eps^d n^{1-s}} + \gamma  \Big)
  \leq
  \mathbb{P}\Big(L_{\eps,m}(S_n)> \mathbb{E}[L_{\eps,m}(S_n)] + \gamma \Big)
  \leq 2\exp\big(-0.09n^{1-s}\gamma^2\big).
\end{align*}
For the second term in (\ref{dec1}), we have
\begin{align}
\nonumber
\mathbb{P}\Big(\Nfar(x) & >  \delta N(x) \Hquad \Big|  \Hquad L_{\eps,m}(S_n) \leq  \frac{1.6}{\eps^d n^{1-s}} +\gamma \Big)
\\&
\leq \mathbb{P}\Big( |B(x,r_n)\cap S_n| < n^{s} \Hquad \Big|
L_{\eps,m}(S_n) \leq  \frac{1.6}{\eps^d n^{1-s}} + \gamma \Big)
\nonumber\\&+
\mathbb{P}\Big(\Nfar(x)> \delta N(x) ,\Hquad  |B(x,r_n)\cap S_n| \geq n^{s} \Hquad \Big|
\nonumber
L_{\eps,m}(S_n) \leq  \frac{1.6}{\eps^d n^{1-s}} + \gamma \Big)\\
& = (*) + (**). \label{dec2}
\end{align}

Since $r_n = \sqrt{d}\eps$, we have $\{|B(x,r_n)\cap S_n| < n^{s} \} \implies \{|C(x)\cap S_n | < n^{s}\}$, where $C(x)$ is the $\eps$-length box containing $x$. Thus,\begin{align*}
(*) \leq \frac{1.6}{\eps^d n^{1-s}} + \gamma. 
\end{align*}

We are left to bound the second term in (\ref{dec2})
\begin{align}
\nonumber
(**) \leq & 
\mathbb{P}\Big( \Nclose(x) < \frac{1}{2}n^{s-\frac{1}{2}} \Hquad \Big|
L_{\eps,m}(S_n) \leq  \frac{1.6}{\eps^d n^{1-s}} + \gamma, \Hquad|B(x,r_n)\cap S_n| > n^{s} \Big)
\\\nonumber +&\mathbb{P}\Big(\Nfar(x)> \delta N(x) ,\Nclose(x) \geq \frac{1}{2}n^{s-\frac{1}{2}}  \Hquad \Big|
L_{\eps,m}(S_n) \leq
\frac{1.6}{\eps^d n^{1-s}} + \gamma, \Hquad|B(x,r_n)\cap S_n| > n^{s}  \Big) \\ \label{dec3}
& = (***) + (****).
\end{align}
Since the algorithm set $m_n =\lfloor\frac{\log{n}}{2\log(\frac{1}{p_{1}})}\rfloor $, we have 
\begin{align*}
  \mathbb{E}\Big[\Nclose(x)\Big| |B(x,r)\cap S_n|>n^s \Big]
  &\geq   p_{1}^{m_n}n^s
  \\& \geq  p_{1}^{\frac{\log{n}}{2\log(\frac{1}{p_{1}})}}n^s
  \\&  \geq \big(2^{\log{p_{1}}}\big)^{\frac{1}{2\log{\frac{1}{p_{1}}}}\log{n}}n^s
  \\& \geq n^{s-\frac{1}{2}}.
\end{align*}
Let $Z\sim\mathrm{Bin}(n^s,p_{1}^{m_n})$. We have 
$\mathbb{E}\Big[\Nclose(x)\Big||B(x,r_n)\cap S_n|>n^s \Big]\geq\mathbb{E}[Z] = n^{s-\frac{1}{2}} $.
In addition, for each $x'\in \Aclose(x)$ we have $\mathbb{P}\big(g_n(x)=g_n(x')\big)\geq p_{1}^{m_n} $, and thus,
invoking the Chernoff bound,
\begin{align*}
  (***)& \leq
  \mathbb{P}(Z<\frac{1}{2} n^{s-\frac{1}{2}})\\
& = \mathbb{P}(Z<\frac{1}{2}\mathbb{E}[Z]) \\
&\leq \exp(-\frac{1}{8}\mathbb{E}[Z]) \\ & \leq \exp(-\frac{1}{8}n^{s-\frac{1}{2}}).
\end{align*}

The last term we have to bound is the second term in (\ref{dec3}). Notice that 
\begin{align*}
  \{\Nfar(x)> \delta N(x) ,\Nclose(x) \geq \frac{1}{2}n^{s-\frac{1}{2}}\}
  \implies \{\Nfar(x)> \frac{\delta}{2} n^{s-\frac{1}{2}}\}.
\end{align*}
In addition, since $p_{1}^2 > p_{2}$, we have

\beq
&\mathbb{E}\big[\Nfar(x)\big] \leq p_{2}^{m_n} n \leq  p_{1}^{2m_n} n \leq p_{1}^{2\big(\frac{\log{n}}{2\log{\frac{1}{p_1}}}-1\big)}n
= p_1^{-2} = O(1) .
\eeq
Since for each $x' \in \Afar(x)$ we have $\mathbb{P}\big(g_n(x)=g_n(x')\big) \leq p_{2}^{m_n}$,  if we let $Z\sim\mathrm{Bin}(n,p_{2}^{m_n})$
then, by Chernoff's bound,
\begin{align*}
(****) \leq \mathbb{P}(Z>\frac{\delta}{2} n^{s-\frac{1}{2}})
\leq 2^{\frac{\delta}{2} n^{s-\frac{1}{2}}}.
\end{align*}
For $s>\frac{1}{2}$ and large enough $n$ s.t. $2e\mathbb{E}\big[\Nfar(x)\big]\leq 2e\mathbb{E}[Z]\leq 2e \leq \frac{\delta}{2} n^{s-\frac{1}{2}}$.

Finally, setting $\gamma = \sqrt{\frac{1}{n^{\frac{1-s}{2}}}}$,
$r_n=\Big(\frac{1.6\sqrt{d}^{d+2}}{n^{1-s}}\Big)^{\frac{1}{d+1}}$ we conclude our proof. 
\end{proof}

\section{Proof of Lemma \ref{lem:hash2}}
Here we show that with high probability, the variable $N(x) \to \infty$. Namely, the number of sample points at each bucket is increasing as $n$ goes to $\infty$. 

\begin{lemma} 
\label{lem:hash2}
Let $x \sim \mathcal{D}_{\mathcal{X}}$ be a test point. Then, for all $M>0$, $\frac{1}{2}<s <1$ the hash table calculated by Algorithm 1 satisfies:
\begin{align*}
  \mathbb{P}\Big( N(x) &< M \Big) \leq
  \exp(-\frac{n^{s-\frac{1}{2}}}{2}+M) +2\exp(-0.09n^{\frac{1-s}{2}})
  + \Big(\frac{1.6}{n^{1-s}\sqrt{d}^d}\Big)^{\frac{1}{d+1}}+ \sqrt{\frac{1}{n^{\frac{1-s}{2}}}}.
\end{align*}
Where, again, the probability is over $S_n,x \sim \mathcal{D}^{n+1}$ and the choice of the function $g_n$.   \\
\end{lemma}

\begin{proof}
Fix $M>0$. Similar to Lemma \ref{lem:hash}, we have
\begin{align}
\nonumber
\mathbb{P}\Big( N(x) < M \Big) \leq
&
\mathbb{P}\Big( N(x) < M \Big|
L_{\eps,m}(S_n) \leq
\frac{1.6}{\eps^d n^{1-s}},\Hquad |B(x,r_n)\cap S_n| > n^{s} \Big) \\\label{le5dec1} &+2\exp(-0.09n^{\frac{1-s}{2}}) +
\frac{1.6 \sqrt{d}^d}{r_n^d n^{1-s}}+ \sqrt{\frac{1}{n^{\frac{1-s}{2}}}}.
\end{align}
We only have to bound the first term in (\ref{le5dec1}). 
Observe that 
\beq
\{N(x) < M \} \implies \{\Nclose(x)<M\}
\eeq
and that $\mathbb{E}\big[\Nclose(x)\big| |B(x,r_n \cap S_n )| > n^s \big] \geq \mathbb{E}[Z] = p_{1}^{m_n} n^{s}=n^{s-\frac{1}{2}}$.
Now
for $Z\sim \mathrm{Bin}(n^s,p_{1}^{m_n})$, if we let $\xi = 1-\frac{M}{\mathbb{E}[Z]}$, then by Chernoff's bound we have, 
\begin{align*}
  \mathbb{P}\Big( \Nclose(x)  < M \Big|
  L_{\eps,m}(S_n)
  \leq
  \frac{1.6}{\eps^d n^{1-s}},|B(x,r_n)\cap S_n| > n^{s} \Big)
  &\leq \mathbb{P}(Z<(1-\xi)\mathbb{E}[Z]) \\
&\leq \exp(-\frac{\xi^2}{2}\mathbb{E}[Z])\\&\leq  \exp(-\frac{n^{s-\frac{1}{2}}}{2}+M).
\end{align*}
\end{proof}

\end{document}